\documentclass[a4paper,11pt]{amsart}

\usepackage{cite}
\usepackage{algpseudocode}
\usepackage{amsmath, amsthm}
\usepackage{textcomp}
\usepackage{amssymb}
\usepackage[all]{xy}
\usepackage{enumerate}
\usepackage{fancyhdr}
\usepackage{lscape}
\usepackage{url}
\usepackage{setspace}
\usepackage{longtable}
\usepackage{tikz}
\usepackage{wasysym}

\usepackage[labelsep=endash]{caption}

\theoremstyle{plain}
\newtheorem{theorem}{Theorem}[section]
\newtheorem{lm}[theorem]{Lemma}
\newtheorem{prop}[theorem]{Proposition}
\newtheorem{cor}[theorem]{Corollary}

\theoremstyle{definition}
\newtheorem{defi}[theorem]{Definition}
\newtheorem{conj}[theorem]{Conjecture}
\newtheorem{ex}[theorem]{Example}

\newtheorem{notation}[theorem]{Notation}

\theoremstyle{remark}
\newtheorem{rmk}{Remark}

\DeclareMathOperator{\Pic}{Pic}
\DeclareMathOperator{\Int}{Int}

\DeclareMathOperator{\Proj}{Proj}
\DeclareMathOperator{\Spec}{Spec}
\DeclareMathOperator{\rank}{rank}
\DeclareMathOperator{\Hom}{Hom}
\DeclareMathOperator{\Cl}{Cl}

\DeclareMathOperator{\Cox}{Cox}

\DeclareMathOperator{\Mob}{Mob}

\DeclareMathOperator{\Sl}{GL}
\DeclareMathOperator{\SL}{SL}

\DeclareMathOperator{\Gl}{GL}
\DeclareMathOperator{\hcf}{gcd}

\newcommand{\F}{\mathcal{F}}

\newcommand{\C}{\mathbb{C}}
\newcommand{\Z}{\mathbb{Z}}
\newcommand{\A}{\mathbb{A}}
\newcommand{\Q}{\mathbb{Q}}

\newcommand{\III}{{\rm III}}
\newcommand{\IV}{{\rm IV}}
\newcommand{\I}{{\rm I}}
\newcommand{\II}{{\rm II}}
\newcommand{\PP}{\mathbb{P}}
\renewcommand{\H}{\text{H}}

\newcommand\qt{{\slash\kern-0.65ex\slash}}

\title{On Pliability of del Pezzo fibrations and Cox rings}

\author{Hamid Ahmadinezhad}

\begin{document}

\begin{abstract} We develop some concrete methods to build Sarkisov links, starting from Mori fibre spaces. This is done by studying low rank Cox rings and their properties. As part of this development, we give an algorithm to construct explicitly the coarse moduli space of a toric Deligne-Mumford stack. This can be viewed as the generalisation of the notion of well-formedness for weighted projective spaces to homogeneous coordinate ring of toric varieties. As an illustration, we apply these methods to study birational transformations of certain fibrations of del Pezzo surfaces over $\PP^1$, into other Mori fibre spaces, using Cox rings and variation of geometric invariant theory. We show that the pliability of these Mori fibre spaces is at least three and they are not rational. \end{abstract}

\maketitle

\tableofcontents

\section{Introduction}

The birational classification of algebraic varieties in dimension three saw a major advancement after the completion of the minimal model programme (MMP for short) in the 1980s. Roughly speaking, given an algebraic variety, MMP produces a simpler model that is either a ``Mori fibre space'', starting from a variety with negative Kodaira dimension, or a ``minimal model'' otherwise. The next step is to study relations among these outputs.

Theoretically, any birational map between Mori fibre spaces can be decomposed as a finite sequence of simpler maps, called {\it Sarkisov links} \cite{corti2, hacon-mcK}.

A major aim of this article is to introduce explicit methods of constructing Sarkisov links, starting from a given Mori fibre space. 
The hope is to push the existing methods to be able to detect many families with finite number of different Mori fibre space structures.

The article is organised in two parts. In the first part we develop some techniques to work with low rank Cox rings and their blow ups. In particular, we show how to write down the Cox ring of the (weighted) blow up of a weighted projective space and also the (weighted) blow up of the projectivization of a weighted vector bundle over $\PP^n$. We also introduce a notion of well-formedness for toric Cox rings, which gives rise to an algorithm to construct the coarse moduli space of a toric Deligne-Mumford stack (TDMS).

In the second part, we apply these techniques to study birational geometry of certain families of del Pezzo surfaces, treated as Mori fibre spaces.  Formally, a variety $X$ with at worst $\Q$-factorial terminal singularities is a Mori fibre space if there exists a morphism $\varphi\colon X\rightarrow Z$ to a normal variety $Z$, of strictly smaller dimension than $X$, such that $-K_X$ is $\varphi$-ample and the relative Picard number is equal to one, that is $\rho(X/Z)=1$. It is crucial in the classification of threefolds to investigate how many different Mori fibre spaces fall in the same birational class, and study their properties. Naturally, this problem is considered up to {\it square birational} equivalence.

\begin{defi} Let $X\rightarrow Z$ and $X^\prime\rightarrow Z^\prime$ be Mori fibre spaces. A birational map $f\colon X\dashrightarrow X^\prime$ is {\it square} if it fits into a commutative diagram\begin{center}$\xymatrixcolsep{3pc}\xymatrixrowsep{3pc}
\xymatrix{
X\ar@{-->}^f[r]\ar[d]& X^\prime\ar[d]\\
Z\ar@{-->}[r]^g&Z^\prime
}$\end{center} where $g$ is birational and, in addition, the map $f_L\!\colon\!X_L\!\dashrightarrow\!X^\prime_L$ induced on generic fibres  is biregular. In this case we say that $X/Z$ 
and $X^\prime/Z^\prime$ are {\it square birational}. We denote this by $X/Z\sim X^\prime/Z^\prime$.
\end{defi}

\begin{defi}[\cite{corti-mella}]\label{pliability}The {\it pliability} of a Mori fibre space $X\rightarrow S$ is the set
\[\mathcal{P}(X\slash S)=\{\text{Mfs } Y\rightarrow T\mid X\text{ is birational to } Y\}\slash\sim\] We sometimes use the term pliability to mean the cardinality of this set, when it is finite.\end{defi}

For many Mori fibre spaces this number is known to be very small or the variety is known to be rational, so that $\mathcal{P}(X\slash Z)=\mathcal{P}(\PP^n)$ is infinite. For instance smooth fibrations of del Pezzo surfaces over $\PP^1$ with degree of the general fibre five or above are known to be rational~\cite{isk}. On the opposite (to rationality), there are many examples with pliability 1. For example, it was proved in \cite{CPR} that if $X$ is a quasi-smooth and general hypersurface of degree $d$ in a weighted projective space $\PP(1,a_1,a_2,a_3,a_4)$ and has only terminal singularities then $\mathcal{P}(X)=1$. This is a generalisation of the famous result of Iskovskikh and Manin for smooth quartic threefold \cite{manin}. Many other similar results have been obtained for threefold Mori fibre spaces, in both directions, for threefolds and also higher dimensional varieties (see for example~\cite{defe, vanya-rigid}). An interesting model with pliability exactly two is constructed as a quartic in $\PP^4$ with a $cA_2$ singular point~\cite{corti-mella}.

In this article we consider a del Pezzo fibration of degree 4 over $\PP^1$ and show that its pliability is at least 3. Then using results from \cite{alexeev} we prove it is not rational. We recall that the degree of a del Pezzo surface is the self-intersection number of its canonical class. Birational geometry of del Pezzo fibrations, considered as Mori fibre spaces, plays an important role in the theory of classification of algebraic varieties in dimension 3. While degree 5 or above imply rationality, one expects more rigidity as the degree drops. There are many results of this type on birational rigidity and nonrigidity for degree 1, 2 and 3 fibrations, see for example \cite{Grinenko, Ahm, BCZ, Pukhlikov, Cheltsov}. The rationality problem in degree 4 is studied in \cite{alexeev, shramov}. It is also shown in \cite{alexeev} that these varieties are birational to conic bundles over the base, hence nonrigidity can be deduced from this. We study the birational behaviour of these varieties modulo this manoeuvre. 
 
The construction of the links between various models, through 2-ray games that generate Sarkisov links, are obtained via Cox embeddings. In \cite{BZ} it was shown how to run type $\III$ or $\IV$ Sarkisov links on a rank 2 Cox rings. This method has been applied to many classes of Mori fibre spaces with Picard number 2, see \cite{BCZ, Ahm}. Some of the models we construct are obtained in this fashion and for others we introduce more general methods of working with higher rank Cox rings and also explain what the maps between these Cox rings look like. In other words, these constructions provide explicit Sarkisov links of type $\I$ or $\II$ in terms of Cox data. Working on a particular case, as an illustration, we obtain the following.

\begin{theorem}\label{main-theorem} Let $\F$ be the five-fold $\F=\Proj_{\PP^1}\mathcal{E}$, where $\mathcal{E}=\mathcal{O}_{\PP^1}\oplus\mathcal{O}_{\PP^1}(1)\oplus\mathcal{O}_{\PP^1}(2)\oplus\mathcal{O}_{\PP^1}(3)\oplus\mathcal{O}_{\PP^1}(3)$. Let $M$ be the class of the tautological bundle on $\F$ and $L$ the class of a fibre (over $\PP^1$). If   hypersurfaces $Q_1\in|-3L+2M|$ and $Q_2\in|-2L+2M|$ in $\F$ are general such that $X=Q_1\cap Q_2$ is a smooth complete intersection in $\F$ then
\begin{enumerate}[(1)]
\item  $\Pic(X)\cong\Z^2$ and $X\rightarrow\PP^1$ is a Mori fibre space with generic fibre del Pezzo surface of degree 4,
\item there exist at least two non-trivial Sarkisov links from $X/\PP^1$ to other Mori fibre spaces,
\item $X$ is not rational.
\end{enumerate}
\end{theorem}

The varieties and the maps appearing in Theorem\,\ref{main-theorem} can be seen in the Diagram\,\ref{easy-diag}. We refer to Section\,\ref{proofs} and Diagram\,\ref{main-diag} for details.

{\small
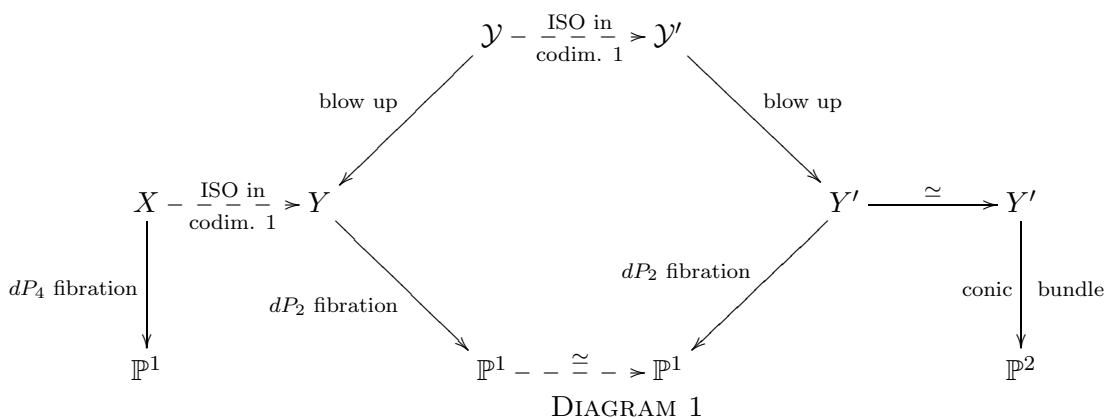
\begin{table}[h]\centering 

\xymatrixcolsep{4pc}\xymatrixrowsep{4pc}
\xymatrix{
&&\mathcal{Y}\ar_{\text{blow up}}[ld]\ar@{-->}^{\text{ISO in}}_{\text{codim. }1}[r]&\mathcal{Y}^\prime\ar^{\text{blow up}}[rd]&&\\
X\ar@{-->}^{\text{ISO in}}_{\text{codim. }1}[r]\ar_{dP_4\text{ fibration}}[d]&Y\ar_{dP_2\text{ fibration}}[rd]&&&Y'\ar_{dP_2\text{ fibration}}[ld]\ar^{\simeq}[r]&Y'\ar_{\text{conic}}^{\text{ bundle}}[d]\\
\PP^1&&\PP^1\ar@{-->}^{\simeq}[r]&\PP^1&&\PP^2
}
\caption{}\label{easy-diag}
\end{table} 
}
In this diagram, by ``ISO in codim.\,$1$'' we mean isomorphism in codimension one and ``$dP_n$ fibration'' stands for a fibration of Mori type with fibres del Pezzo surfaces of degree $n$, and the symbol $\simeq$ denotes an isomorphism. The varieties $Y$ and $Y'$ are non-isomorphic but square birational, over $\PP^1$.

This result, and some experience in working with del Pezzo fibrations, lead to the following natural conjecture. In fact, the  relation between the condition that $-K_X\notin\Int(\overline{\Mob(X)})$ and absence of birational maps (Sarkisov links of type $\III$ or $\IV$ to be precise) has been observed before (see for example \cite{BCZ, Grinenko1, Grinenko2}). This conjecture serves as an analogue for del Pezzo fibrations in degree\,$4$.

\begin{conj} Let $X$ be a  smooth threefold Mori fibre space over $\PP^1$, with fibres being smooth del Pezzo surfaces of degree $4$. Then, except the birational maps obtained by blowing up a section, there is no Sarkisov link starting from $X$ to another Mori fibre space if and only if $-K_X\notin\Int(\overline{\Mob(X)})$.
\end{conj}

Let us recall that a divisor is mobile if its linear system has no base component, and $\Mob(X)$ is the convex cone in $N^1(X)$ generated by the classes of mobile divisors. It is easy to check, by computations in Section~\ref{proofs}, that $-K_X\in\Int(\overline{\Mob(X)})$ for the threefold considered in Theorem~\ref{main-theorem}, hence there are other Sarkisov links starting from $X$.

Later in Subsection\,\ref{bl-up-wps}, Example\,\ref{elliptic}., we employ our explicit constructions to recover some elliptic fibrations on a Fano threefold. This is to illustrate that the result of \cite{vanya-jihun}, on classification of elliptic and $K3$ fibrations on Fano threefold hypersurfaces, can be recovered with global methods as was the aim of Ryder\,\cite{Ryder}, who obtained partial results in this direction.

The structure of the article is as follows. In Sections~\ref{stack} we study well-formedness of the homogeneous coordinate ring of toric varieties. In Section~\ref{blow-up} we show how to work with blow ups of low rank Cox rings. Sections~\ref{stack} and \ref{blow-up} can be considered independently of this article and the results are quite general and do not restrict only to the cases studied in this article. Among applications of these methods is the description of the starting point of type $\I$ and $\II$ Sarkisov links, that we apply. Section~\ref{stack} explains how some of the blow up varieties can be modified to simpler ones, isomorphically. This generalizes the notion of well-formedness of weighted projective spaces \cite{fle} to that of Cox rings. Equivalently, it is an explicit method to find Cox rings of coarse moduli of TDMS \cite{fantechi}. In Section \ref{proofs} we describe explicitly the links between varieties under study. Tools provided in earlier sections will be used frequently in the proofs.

All varieties and stacks in this article are projective. Results in Sections \ref{stack} and \ref{blow-up} hold in any field of characteristic zero. In Section \ref{proofs} varieties are considered over the field of complex numbers. 

{\bf Acknowledgement.} The author would like to thank Gavin Brown, Ivan Cheltsov, Alessio Corti, Anne-Sophie Kaloghiros and Francesco Zucconi for useful discussions. Special gratitude goes to the anonymous referees whose valuable remarks helped this paper find its current shape. This work was partially supported by the Austrian Science Fund (FWF): P22766-N18.

\section{Well-formedness and stacky models}\label{stack}
The homogeneous coordinate ring of a toric variety was introduced in \cite{Cox} in order to view toric varieties as certain quotient spaces. It was later generalised to projective varieties by Hu and Keel \cite{hu} and inherited the name {\it Cox ring} in that article. In \cite{borisov}, Borisov, Chen and Smith introduced {\it stacky fans}, to define TDMS, and formulated the quotient construction of these varieties as quotient stacks, following \cite{Cox}. It was also shown how to recover the coarse moduli space of a TDMS, as a toric variety, via the stacky fan. A geometric approach to TDMS was carried out in \cite{fantechi}, and the properties of the quotient and relations to the coarse moduli space was studied in more depth. In this section we show, given the quotient stack corresponding to a TDMS, how to recover its coarse moduli space (as a Cox ring of a toric variety).
It is also shown that this construction is a generalisation of the method of ``well forming'' a weighted projective space (see \cite{fle}).

\subsection{Well-formedness of weighted projective spaces}
Weighted projective spaces have been studied extensively, and the well-formedness property, as in \cite{dol} and \cite{fle} plays an important role in the basic theory.

A weighted projective space, denoted by $\PP(a_0,\dots,a_n)$, for positive integers $a_0,\dots,a_n$,  is defined by the geometric quotient of $\C^{n+1}-\{0\}$ when acted on by $\C^*$ via
\[\lambda.(x_0,\dots,x_n)\mapsto(\lambda^{a_0}x_0,\dots,\lambda^{a_n}x_n),\quad\text{for }\lambda\in\C^*.\]
In other words, $\PP(a_0,\dots,a_n)=\Proj\C[x_0,\dots,x_n]$, where $\C[x_0,\dots,x_n]$ is $\Z$-graded with $\deg(x_i)=a_i$.

The weighted projective space $\PP(a_0,\dots,a_n)$ is {\it well-formed} if $\hcf(a_0,\dots,\hat{a_i},\dots,a_n)=1$ for all $0\leq i\leq n$. The well-formed model of a given quotient $\PP(a_0,\dots,a_n)$ is obtained in two steps:
\begin{enumerate}[Step\,1.]
\item Remove generic stabilisers \cite[Lemma~5.5]{fle}: Find $a=\hcf(a_0,\dots,a_n)$, then replace $\PP(a_0,\dots,a_n)$ by $\PP(b_0,\dots,b_n)$, where $b_i=\frac{a_i}{a}$.
\item Remove quasi-reflections \cite[Lemma~5.7]{fle}: For each $b_i$ find $\mathbf{b_i}=\hcf(b_0,\dots,\hat{b_i},\dots,b_n)$ and replace $\PP(b_0,\dots,b_n)$ by \[\PP(\frac{b_0}{\mathbf{b_i}},\dots,\frac{b_{i-1}}{\mathbf{b_i}},b_i,\frac{b_{i+1}}{\mathbf{b_i}},\dots,\frac{b_n}{\mathbf{b_i}}).\]
\end{enumerate}

These follow from the fact that for a given set of positive integers defining the weights, and positive integers  $\alpha$ and $\beta$ the following holds:
\[\PP(a_0,\dots,a_n)\cong\PP(\alpha a_0,\dots,\alpha a_n)\cong\PP(a_0,\beta a_1,\dots,\beta a_n)\quad\text{when }\gcd(a_0,\beta)=1\]
and this is exactly why one is permitted to do the process above and obtain isomorphic quotients.

\subsection{Toric varieties and TDMS as quotients}
Let $X$ be a toric variety of dimension $d$ determined by the fan $\Delta$ in $N\cong\Z^d$ and denote the set of $1$-dimensional cones in $\Delta$ by $\Delta(1)$. Denote by $M=\Hom(N,\Z)$ the dual lattice and assume $\Delta(1)=\{\rho_1,\dots,\rho_n\}$. Suppose $b_i\in\rho_i\cap N$ are the unique generators of each row and $\{b_1,\dots,b_n\}$ spans $N_\Q=N\otimes_\Z\Q$.

One has the following exact sequence, see (\cite{ful}, \S3.4).

\begin{equation}\label{exact-seq1}\xymatrixcolsep{2.8pc}\xymatrixrowsep{2.8pc}
\xymatrix{
0\ar[r]&M\ar[r]^B&\Z^n\ar[r]^A& \Cl(X)\ar[r]& 0}\end{equation}
where $B=[b_1\cdots b_n]^T$, see \cite[14.3]{Cox-book} and \cite{Cox}. The matrix $A$ is simply the Gale dual of $B_\Q\colon M_\Q\rightarrow\Q^n$, and it is an $r\times n$ integer matrix, unique up to an action of $\Gl(r,\Z)$ corresponding to the unimodular coordinate transformations of $\Z^r$, where $r=n-d$. The rows of $A$ determine the $r$ integer relations among the $b_i$s.

The Cox ring of $X$ is defined to be $\Cox(X)=\C[x_1,\dots,x_n]$ with a grading: consider the action of $G=\Hom_\Z(\Cl(X),\C^*)$ on $(\C^n-V(I))$, where $I\subset\Cox(X)$ is the irrelevant ideal associated to $\Delta$, see \cite[\S1]{Cox}.

Assume that the fan $\Delta$ is simplicial, so that the quotient $(\C^n-V(I))\qt G$ is geometric \cite[Theorem~2.1 (iii)]{Cox}. For the purpose of this section this is a fine assumption by \cite[Remark~3.4]{borisov}. Also assume that $G\cong(\C^*)^r$. The aim in this section is to study the quotient space and give algorithms to remove generic stabilisers and quasi-reflections. If $G$ is not a torus then it has a torsion part for which removing generic stabilisers and quasi-reflections are trivial.

The action of $G$ defines a grading on $\Cox(X)$ by the columns of the matrix $A$, as follows. By applying the functor $\Hom(-,\C^*)$ to the short exact sequence~(\ref{exact-seq1}) one obtains
 \begin{equation}\label{exact-seq2}\xymatrixcolsep{2.8pc}\xymatrixrowsep{2.8pc}
\xymatrix{
1\ar[r]&G\ar[r]^{A^*}&(\C^*)^n\ar[r]& \mathbb{T}\ar[r]& 1,}\end{equation}
where $\mathbb{T}$ is the torus acting on $X$.
The action of $G$ on $\C^n$ is the extension of the action on $(\C^*)^n$ above and is identified by the matrix $A=(a_{ij})$ above in the following way
\[(\lambda_1,\dots,\lambda_r)\cdot x_j\mapsto \prod_{i=1}^r\lambda_i^{a_{ij}}x_j.\]
Hence the data $(\C^n, I, A)$ identifies the original toric variety. 
\begin{notation}We denote this quotient space by $X=(I,A)$, where $\C^n$ is dropped as it can be realised from $I$ and $A$.\end{notation}

\paragraph*{\bf Stacky models} Following~\cite{borisov}, the stacky fan $\bold\Delta$ is a triple $(N, \Delta, \beta)$, where $N$ is a finitely generated abelian group, $\Delta$ is a simplicial fan in $N\otimes_\Z\Q$ with $n$ rays and a map $\beta\colon\Z^n\rightarrow N$, where $\beta(e_i)$, the image of standard basis in $\Z^n$, generates the rays in $\Delta$. With the generalisation of Gale duality as in \cite[\S.2]{borisov}, the TDMS $\mathcal{X}$ associated to the stacky fan $\bold\Delta$ can be recovered as the quotient $[Z/G]$, where $Z=\C^n-V(I)$ as before, and the action of $G$ is defined by the Gale dual of $\beta$ similar to the toric case, see \cite[\S.3]{borisov}. Hence the TDMS $\mathcal{X}$ can also be recovered from the (stacky) data $(\C^n,I,A)$. Note that the difference with the non-stacky construction is incorporated fully in $A$. 
\begin{notation} Similar to the previous case we denote this quotient by $\mathcal{X}=[I,A]$.
\end{notation}
Let us recall that the coarse moduli space of $\mathcal{X}$ is the toric variety associated to the non-torsion part of $N$ and $\Delta$ (\!\cite[Proposition~3.7]{borisov}).

The aim in this section is to recover the coarse moduli space of $\mathcal{X}$ given the information $(\C^n,I,A)$, without constructing the fans.

\subsection{Well formed Cox rings}

\begin{defi} For a toric variety as above, or a TDMS, we define its rank to be the rank of the matrix $A$.\end{defi}

\begin{lm}\label{gliso} Let $X=(I,A)$, so that the quotient is geometric, and $B=gA$ for some $g\in \Gl(r,\mathbb{Q})$ with integer entries. Suppose $X'$ is the quotient of $\C^n-V(I)$ by the group $(\C^*)^r$ so that its action is defined similar to that of $X$ but via $B$. Then $X$ is 
isomorphic to $X'$ as quotient varieties.
\end{lm}
\begin{proof} We give an explicit and set theoretic proof. $X$ and $X'$ are defined by
\[X=(\C^n-V(I))\slash G_A,\qquad X'=(\C^n-V(I))\slash G_B, \]
where $G_A\cong G_B\cong(\C^*)^r$. If we denote $A=(a_{ij})$ and $B=(b_{ij})$, then for $(\lambda_1,\dots,\lambda_r)\in G_A$ and $(\gamma_1,\dots,\gamma_r)\in G_B$, the actions are the following:
\[G_A:\qquad(\lambda_1,\dots,\lambda_r).(x_1,\dots,x_n)\mapsto(\prod_{i=1}^r\lambda_i^{a_{i1}}x_1,\dots,\prod_{i=1}^r\lambda_i^{a_{in}}x_n)\]
\[G_B:\qquad(\gamma_1,\dots,\gamma_r).(x_1,\dots,x_n)\mapsto(\prod_{i=1}^r\gamma_i^{a_{i1}}x_1,\dots,\prod_{i=1}^r\gamma_i^{a_{in}}x_n)\]
Let $(\text{x})$ and $(\text{y})$ be two vectors in $\C^n$. Let us denote by $(\text{x})\sim_A(\text{y})$ if $(\text{x})$ and $(\text{y})$ are in the same orbit of the action by $G_A$, and similarly for $(\text{x})\sim_B(\text{y})$. The aim is to show 
\[(\text{x})\sim_A(\text{y})\quad \text{ if and only if }\quad (\text{x})\sim_B(\text{y}).\]
If $(\text{x})\sim_B(\text{y})$, then there exists $(\gamma_1,\dots,\gamma_r)\in(\C^*)^r$ such that 
\[(y_1,\dots,y_n)=(\prod_{i=1}^r\gamma_i^{b_{i1}}x_1,\dots,\prod_{i=1}^r\gamma_i^{b_{in}}x_n).\]
To prove $(\text{x})\sim_A(\text{y})$, we must find $(\lambda_1,\dots,\lambda_r)\in(\C^*)^r$ such that
\[(y_1,\dots,y_n)=(\prod_{i=1}^r\lambda_i^{a_{i1}}x_1,\dots,\prod_{i=1}^r\lambda_i^{a_{in}}x_n).\]
This follows from $b_{ij}=\sum_kg_{ik}a_{kj}$, if we put $\lambda_i=\gamma_1^{g_{i1}}\dots\gamma_r^{g_{ir}}$.\\
Proof for the only if part is very similar and it is done by replacing $g$ by $g^{-1}$.
\end{proof}

\begin{defi}Let $M\in \mathcal{M}_{r\times n}(\mathbb{Z})$ be a rank $r$ matrix ($r\leq n$). Suppose $m_1,\dots, m_s$ are all the non-zero $r\times r$ minors of $M$ and let $d_M=\gcd(|m_1|,\dots,|m_s|)$. The matrix $M$ is called {\it standard} if  $d_M=1$.\end{defi}

\begin{lm}\label{dpre} For any rank $r$ matrix $M\in\mathcal{M}_{r\times n}(\Z)$, there exist matrices $g\in \Gl(r,\Q)\cap \mathcal{M}_{r\times r}(\Z)$ and $N\in \mathcal{M}_{r\times n}(\Z)$ such that $M=gN$ and $N$ is a standard matrix of rank $r$.\end{lm}

We try to remove every factor of $d_M$ by multiplying $M$ with a matrix whose inverse is in $\Gl(r,\mathbb{Q})\cap \mathcal{M}_{r\times r}(\mathbb{Z})$.

\begin{proof}We induct on the value $d_M$. If $d_M=1$, there is nothing to show. Suppose $d_M>1$, and let $p$ be a prime factor of $d_M$. Denote by $\overline{M}\in\mathcal{M}_{r\times n}(\mathbb{F}_p)$ the reduction modulo $p$ of $M$. Find 
$\overline{g}\in\SL_r(\mathbb{F}_p)$ such that $\overline{g}\overline{M}$ is the row echelon form up to constant. Note that since $d_M$ is divisible by $p$, $\overline{M}$ has the rank strictly smaller than $r$; this in turn implies that all the entries of the $r$-th row of $\overline{g}\overline{M}$ are $0$. By the surjectivity of the map $\SL_r(\Z)\rightarrow\SL_r(\mathbb{F}_p)$, we find a lift $g$ of $\overline{g}$. Since the $r$-th row of $gM$ is divisible by $p$, if we define $N\in\mathcal{M}_{r\times n}(\Z)$ as
\[gM= \left(\begin{array}{cc}
I_{r-1}&0\\
0&p\end{array}\right)N,\]
then $d_N = \frac{1}{p}d_{gM} = \frac{1}{p}d_M$ holds. By applying the induction hypothesis to $N$, we obtain the conclusion.
\end{proof}

\begin{prop}\label{LA} Let $A\in\mathcal{M}_{r\times n}(\Z) $ be a matrix of rank $r$. The following are equivalent.
\begin{enumerate}[(i)]
\item $A\colon\Z^n\rightarrow\Z^r$ is surjective.
\item $\wedge^rA\colon\wedge^r\Z^n\rightarrow\wedge^r\Z^r\cong\Z$ is surjective.
\item $A$ is standard.
\end{enumerate}\end{prop}
\begin{proof}This is straightforward linear algebra.\end{proof}

\begin{defi}Let $A$ be a standard $r\times n$ matrix with integer entries. Suppose $A_k$ is an $r\times (n-1)$ matrix obtained by removing the $k$-th column of $A$. The matrix $A$ is called {\it well-formed} if every $A_k$ ($1\leq k\leq n$) is standard.\end{defi}

\begin{lm}\label{pre1dspre} Let $X=(I,A)$ be a toric variety defined by an irrelevant ideal $I$ and an $r\times n$ matrix $A=(a_{ij})$, as before. Assume $q\neq 1$ is a positive integer such that $q\mid a_{1j}$ for $j>1$ but $q\nmid a_{11}$. Define the matrix $B=(b_{ij})$ by $b_{i1}=q.a_{i1}$ and $b_{ij}=a_{ij}$ for $j>1$. Then $X_B=(\C^n-V(I))/G_B$ is isomorphic to $X$ as schemes.\end{lm}

\begin{proof} The matrix $A$ can be arranged so that all elements in the last row are zero or divisible by $q$ except $a_{r1}$; This is done as in the proof of Lemma\,\ref{dpre} after a multiplication by a suitable $g$. Suppose $R=\C[x_1,\dots,x_n]$ is the Cox ring of $X$ graded by $A$. Columns of $A$, denoted by $C_1,\dots,C_n$, considered as lattice points in $\Z^r$, define the grading on $R$ and determine various cones associate to $X$, in particular the ample cone; see for example Section\,7 in\,\cite{Hausen}. Take $D=\sum\alpha_iC_i$ such that $D$, in the ample cone of $X$, corresponds to a very ample divisor. Note that $D$ can be chosen so that $q$ does not divide $\alpha_1$. Let $R_D$ be the ring generated by monomials of degree $D$, then $X\cong X_D=\Proj R_D$. 

Now consider the ring $R_{qD}$. Let $x_1^{a_1}\dots x_n^{a_n}$ be a monomial in $R_{qD}$. If $a_1\neq 0$ there is a positive integer $m$ such that
\[a_1C_1+\dots+a_nC_n=mqD.\]

In particular, $a_1a_{r1}=q\alpha$ for some non-zero integer $\alpha$, so $q$ divides $a_1$. Therefore $x_1$ appears in $R_{qD}$ only to powers of $x_1^q$. Hence $R_qD\cong R'_{qD}$, where $R'_{qD}$ is constructed like $R_{qD}$ for the ring $R'=\C[x_1^q,x_2,\dots,x_n]$, with the grading $\deg(x_i)=C_i$. But this is the coordinate ring of $X_B$ (graded by the columns of the matrix $B$).\end{proof}

Note that this lemma implies for a standard matrix $A$ that is not well-formed, one can perform the process of Lemma~\ref{dpre} to obtain a well-formed model, which provides isomorphic quotient schemes. In fact, the following process produces a well-formed model for a matrix in $\mathcal{M}_{r\times n}(\Z)$.

\paragraph*{\bf Well forming process} Suppose $A$ is standard but not well-formed. Without loss of generality, we can assume that $A_1$ is not well-formed. Find a matrix $g\in\SL_r(\Z)$ as in the proof of Lemma\,\ref{dpre} so
that the top row of $gA_1 = (gA)_1$ is divisible by a prime number $q$. Then we are in the
situation of Lemma\,\ref{pre1dspre} and can multiply the first column of $gA$ by $q$, without changing
the corresponding GIT quotient. Finally divide the first row of $gA$ by $q$ and let $C$ denote
the new matrix. It is easy to check that $d_{C_i} =d_{A_i} =1$ for $i>1$ and $d_{C_1} = \frac{1}{q}d_{A_1}$. Replace $A$ with $C$ and repeat the procedure.

The following is an immediate corollary of the construction of well-formed models and \cite[Proposition\,3.7]{borisov}.
\begin{cor} Given TDMS $\mathcal{X}=[I,A]$, its coarse moduli space is the toric variety $X=(I,A')$, where $A'$ is the well-formed matrix of $A$.
\end{cor}

\begin{defi} Following \cite{fle}, the quotient space, or the toric variety or its Cox ring, defined by $(I,A)$ is called {\it well-formed} if its defining matrix $A$ is well-formed.\end{defi}

\begin{ex} The notion of well-formedness of weighted projective spaces is a special case of the well-formedness for Cox rings.\end{ex}

Let us demonstrate the process of well forming for an example. As a warm up for the following sections and the techniques used therein, we fully explain the variation of geometric invariant theory (VGIT) for this example. The notations used in this example will be used throughout this article.

\begin{ex}\label{F2}Consider the following matrix:
\[A=\left(\begin{array}{ccccc}
3&3&3&0&-2\\
1&1&1&2&0
\end{array}\right)\]
The set of nonzero determinants of its $2\times 2$ minors is $\{2,4,6\}$, and hence it is not standard as $d_A=2$. The standard model can be obtained by multiplying the following $2\times 2$ matrix
\[\left(\begin{array}{cc}1&0\\-1&1\end{array}\right)\in\Sl(2,\Z)\]
and then removing the factor of $2$ from the second row, which results in
\[A^s=\left(\begin{array}{ccccc}
3&3&3&0&-2\\-1&-1&-1&1&1
\end{array}\right)\]
Following out notation, $A_i$ is the matrix obtained from $A$ by removing the $i$-th column, we have that $d_{A^s_1}\!=\!\cdots\!=\!d_{A^s_4}\!=\!1$ and $d_{A^s_5}=3$. Applying Lemma~\ref{pre1dspre} we multiply the $5$-th column of $A^s$ we get
\[\left(\begin{array}{ccccc}
3&3&3&0&-6\\-1&-1&-1&1&3
\end{array}\right)\]
This matrix is not standard and its standard model can be obtained by removing a factor or $3$ from the first row
\[\left(\begin{array}{ccccc}
1&1&1&0&-2\\-1&-1&-1&1&3
\end{array}\right)\]

For ease of notation we perform an extra change of coordinates and bring the final matrix to the following form. This is done by a multiplication by the $2\times 2$ matrix 
\[\left(\begin{array}{cc}1&0\\1&1\end{array}\right)\in\Sl(2,\Z)\]
or equivalently by adding the second row of the matrix to the first row.
\[A'=\left(\begin{array}{ccccc}
1&1&1&0&-2\\0&0&0&1&1
\end{array}\right)\]
\end{ex}

\paragraph*{\bf A toric model:} Let $X=(I,A')$, where $A'$ is the matrix above and $I=(x,y,z)\cap(t,u)$ is the irrelevant ideal in the Cox ring $R=\C[x,y,z,t,u]$, graded by $A'$. Fan of $X$, as a toric variety, consists of $5$ one-dimensional cones generated by primitive vectors
\[r_1=(1,0,0), r_2=(0,1,0), r_3=(1,1,2), r_4=(-1,-1,-1), r_5=(1,1,1)\]
in $\Z^3$, and $6$ maximal cones 
\[\sigma_1=\left<r_1,r_2,r_4\right>,\quad\sigma_2=\left<r_1,r_2,r_5\right>,\quad\sigma_3=\left<r_1,r_3,r_4\right>\]
\[\sigma_4=\left<r_1,r_3,r_5\right>,\quad\sigma_5=\left<r_2,r_3,r_4\right>,\quad\sigma_6=\left<r_2,r_3,r_5\right>\]
Note that $r_1,\dots,r_5$ form the Gale transform of $A'$.

\paragraph*{\bf VGIT} Considering the action of $G\cong(\C^*)^2$ on $\C^5$ via $A'$, the geometric invariant theory chambers are the following
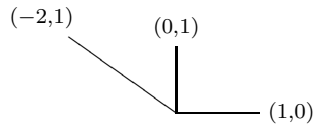
\[\xygraph{
!{(0,0) }="a"
!{(1.2,0) }*+{\scriptstyle{(1,0)}}="b"
!{(0,0.9) }*+{\scriptstyle{(0,1)}}="c"
!{(-1.4,1) }*+{\scriptstyle{(-2,1)}}="d"
"a"-"b"  "a"-"c" "a"-"d" }\]
Choosing a character from the interior of the cone generated by $(1,0)$ and $(0,1)$ constructs a variety isomorphic to $X$ (this is the ample cone of $X$), while a character from the cone generated by $(1,0)$ corresponds to $\PP^2$. Similarly the ray $(0,1)$ corresponds to $\PP(1,1,1,2)$. Let us see this in more details. The quotient corresponding to the character $(1,0)$ is, equivalently, the variety
\[X^{(1,0)}=\Proj\bigoplus_{n\geq 1}\H^0(X,\mathcal{O}_X(n,0))=\Proj\C[x,y,z]\cong\PP^2\]
and similarly for $(0,1)$
\[X^{(0,1)}=\Proj\bigoplus_{n\geq 1}\H^0(X,\mathcal{O}_X(0,n))=\Proj\C[t,ux^2,uxy,uy^2,uxz,uyz,uz^2]\]
which is the second Veronese embedding of $\PP(1,1,1,2)$. 

Clearly the map $X\rightarrow X^{(1,0)}$ is a fibration over $\PP^2$ with $\PP^1$ fibres. The map $X\rightarrow X^{(0,1)}\cong\PP(1,1,1,2)$ is given, in coordinates, by
\[(x,y,z,t,u)\mapsto(u^\frac{1}{2}x:u^\frac{1}{2}y:u^\frac{1}{2}z:t)\]
That shows that $X$ is the blow up of $\PP=\PP(1,1,1,2)$ at the singular point $(0:0:0:1)$.

Let us have a closer look at this variety and the map above. $\PP$ is an orbifold with a terminal cyclic quotient singularity of type $\frac{1}{2}(1,1,1)$. Denote the eigencoordinates of $\PP$ by $x,y,z,t$. The projective space $\PP$ is covered by $4$ open affine patches, three of them are $U_x\cong U_y\cong U_z\cong\C^3$, where $U_x$, for example, is the Zariski open subset $x\neq 0$, and the fourth patch is $U_t=\C^3\slash \Z_2$. The action of $\Z_2$ on $\C^3$ is given by $(\bar{x},\bar{y},\bar{z})\mapsto(\epsilon\bar{x},\epsilon\bar{y},\epsilon\bar{z})$, for $\epsilon$ a second root of unity, and is traditionally denoted by $\frac{1}{2}(1,1,1)$.

It is easy to check that the fan of $\PP$ (as a toric variety) consists of $4$ rays $r_1,r_2,r_3, r_4$ above and its maximal cones are
\[C_1=\left<r_1,r_2,r_3\right>,\,C_2=\left<r_1,r_3,r_4\right>,\,C_3=\left<r_1,r_2,r_4\right>\text{ and }C_4=\left<r_2,r_3,r_4\right>\]

Hence the fan of $X$ is obtained by adding the ray $r_5$ and subdividing the cone $C_1$, and $X$ is the resolution of $\PP$.

This VGIT for a rank $2$ toric variety is equivalent to running the $2$-ray game for it. See \cite{corti2} and \cite{Corti1} for an introduction to Sarkisov program and $2$-ray game, and Section\,3 in \cite{BZ} for this toric correspondence.

\section{Blow ups of low rank Cox rings and Sarkisov links}\label{blow-up}
We begin this section by considering a special class of rank 2 toric varieties. The goal is to understand their singularities and to construct explicit tools to write down the Cox rings of their (toric, weighted) blow ups. In \cite{BZ} it was shown how the Cox data of a rank 2 Cox ring changes as one runs the 2-ray game to obtain type $\III$ or $\IV$ Sarkisov links. Another goal in this section is to understand what happens, in terms of Cox data, when one runs the 2-ray game by starting a blow up on the rank 2 Cox ring, which, in good situations, leads to type $\I$ or $\II$ Sarkisov links.

\begin{defi}\label{weighted bundle-chap3}A {\it weighted bundle over $\PP^n$} is a rank 2 toric variety $\F=(A,I)$ defined by
\begin{enumerate}[(i)]
\item $\Cox(\F)=\C[x_0,\dots,x_n,y_0,\dots,y_m]$.
\item The irrelevant ideal of $\F$ is $I=(x_0,\dots,x_n)\cap(y_0,\dots,y_m)$.
\item and the $(\C^*)^2$ action on $\C^{n+m+2}$ is given by
\[A=\left(\begin{array}{ccccccc}
1&\dots&1&-\omega_0&-\omega_1&\dots&-\omega_m\\
0&\dots&0&1&a_1&\dots&a_m\end{array}\right),\]
where $\omega_i$ are non-negative integers and $\PP(1,a_1,\dots,a_m)$ is a well-formed weighted projective space.
\end{enumerate}\end{defi}
We denote this quotient by $\F_\PP(\omega_0,\dots,\omega_m)$ or sometimes simply by $\F$, when there is no ambiguity.

Note that the weighted bundle $\F$ defined in Definition~\ref{weighted bundle-chap3} is well-formed because the weighted projective space $\PP(1,a_1,\dots,a_m)$ is well-formed.

The following lemma constructs the fan associated to the weighted bundle in Definition~\ref{weighted bundle-chap3}. For a general description of the correspondence between the fan and the GIT construction see, for example, \cite[Section\,5.1]{Cox-book}.

\begin{theorem}\label{fan-rank2} Let ${\beta_1,\dots,\beta_m,\alpha_1,\dots,\alpha_n}$ be the standard basis of $\Z^{n+m}$. Suppose $\alpha_0$ and $\beta_0$ are the following vectors in $\Z^{n+m}$.
\[\beta_0=-{\sum_{i=1}^ma_i\beta_i}\quad ,\quad \alpha_0=-{\sum_{j=1}^n\alpha_j}+{\sum_{i=0}^m\omega_i\beta_i},\]
where $\omega_i$ are non-negative integers. Let $\sigma_{rs}=\left<\beta_0,\dots,\hat{\beta_r},\dots,\beta_m,\alpha_0,\dots,\hat{\alpha_s},\dots,\alpha_n\right>$ be the cone in $\mathbb{Z}^{n+m}$ generated by $\beta_0,\dots,\hat{\beta_r},\dots,\beta_m$ and $\alpha_0,\dots,\hat{\alpha_s},\dots,\alpha_n$, where $\alpha_s$ and $\beta_r$ are omitted. If we denote $\Sigma$ for the fan in $\Z^{n+m}$ generated by maximal cones $\sigma_{rs}$ for all $0\leq r\leq n$ and $0\leq s\leq m$, then $\F\cong T(\Sigma)$.\end{theorem}
\begin{proof} We compute the GIT construction of this fan following the recipe of Cox given in \cite{cox1}~\S10. By assumption, rays $\alpha_0,\dots,\alpha_n,\beta_0,\dots,\beta_m$ in $N=\Z^{m+n}$ form $\Delta(1)$, the set of 1-dimensional cones in $\Sigma$. Let us associate the variables $x_0,\dots,x_n,y_0,\dots,y_m$ to these rays. For a given maximal cone $\sigma$, define $x^\sigma$ to be the product of all variables not coming from rays of $\sigma$. But maximal cones in $\Sigma$ are exactly $\sigma_{rs}$, which immediately implies $x^{\sigma_{rs}}=x_sy_r$. The irrelevant ideal is given by 
\[I=(x^\sigma\mid \sigma\in\Sigma\text{ is a maximal cone})=(x_sy_r\mid 0\leq s\leq n\text{ and }0\leq r\leq r).\]

so that $I=(x_0,\dots,x_n)\cap(y_0,\dots,y_m)$.

In order to describe the GIT construction of $T(\Sigma)$ we must find the group $G$ such that
\[T(\Sigma)\cong(\Spec[x_0,\dots,x_n,y_0,\dots,y_m]-V(I))\slash G,\]

where $G\subset(\C^*)^{m+n+2}$ is defined by
\[G=\{(\mu_0,\dots,\mu_{n},\lambda_0,\dots,\lambda_{m})\in(\C^*)^{m+n+2}\mid \prod_{i=0}^{n}\mu_i^{\left<e_k,\alpha_i\right>}.\prod_{j=0}^{m}\lambda_j^{\left<e_k,\beta_j\right>}=1,\text{ for all }k\},\]
and $e_1,\dots,e_{m+n}$ form the standard basis of $\Z^{m+n}$. But the standard basis of $\Z^{m+n}$, by assumption, is $\{\alpha_1,\dots,\alpha_n,\beta_1,\dots,\beta_m\}$. 

Computing this set implies that $(\mu_0,\dots,\mu_n,\lambda_0,\dots,\lambda_m)\in G$ if and only if
\[\mu_i.\mu_0^{\left<\alpha_0,\alpha_i\right>}.\lambda_0^{\left<\beta_0,\alpha_i\right>}=1\quad\text{ and }\quad \lambda_j.\mu_0^{\left<\alpha_0,\beta_j\right>}.\lambda_0^{\left<\beta_0,\beta_j\right>}=1\quad\forall\, i,j.\]
In other words, $\lambda_0$ and $\mu_0$ determine all other $\lambda_j$ and $\mu_i$. Therefore the group $G$ is isomorphic to $(\C^*)^2$ and the action on coordinate variables is defined by
\[((\mu,\lambda).x_0)=\mu x_0\qquad  ((\mu,\lambda).x_i)=\mu^{-\left<\alpha_0,\alpha_i\right>}\lambda^{-\left<\beta_0,\alpha_i\right>}x_i,\]
\[((\mu,\lambda).y_0)=\lambda y_0 \qquad  ((\mu,\lambda).y_j)=\mu^{-\left<\alpha_0,\beta_j\right>}\lambda^{-\left<\beta_0,\beta_j\right>}y_j.\]
 
In other words, $(\C^*)^2$ acts on $\C[x_0,\dots,x_n,y_0,\dots,y_m]$ by the matrix
\[B=\left(\begin{array}{ccccccc}
1&\dots&1&0&\omega_0a_1-\omega_1&\dots&\omega_0a_m-\omega_m\\
0&\dots&0&1&a_1&\dots&a_m\end{array}\right).\]

We have shown so far that $T(\Sigma)\cong \mathfrak{X}(B,I)$. Multiplying $B$ on the left by the matrix 
\[\left(\begin{array}{cc}1&-\omega_0\\0&1\end{array}\right)\in\Sl(2,\Z),\]
together with Lemma~\ref{gliso} proves that $\F\cong T(\Sigma)$. \end{proof}

\begin{rmk}In \cite{M1} Chapter~2, Reid gives a detailed analysis of rational scrolls, which, in our setting, are the weighted bundles over $\PP^1$, with weights 1 only. In fact these are the smooth weighted bundles.\end{rmk} 

\begin{prop}\label{singularity-F} A well-formed weighted bundle $\F$, defined in Definition~\ref{weighted bundle-chap3}, is covered by $(n+1)(m+1)$ patches, each of them isomorphic to a quotient of $\mathbb{C}^{n+m}$ by a cyclic group $\mathbb{Z}_k$, for some positive integer $k$.\end{prop}
\begin{proof} We construct the patches $\mathcal{U}_{ij}$ for $0\leq i\leq n$ and $0\leq j\leq m$. Note that at the toric level, $\mathcal{U}_{ij}=(x_iy_j\neq 0)$ corresponds to the maximal cone $\sigma_{ij}$ as in Proposition~\ref{fan-rank2}.

\[\mathcal{U}_{ij}=\Spec\mathbb{C}[x_0,\dots,x_n,y_0,\dots,y_m,x_i^{-1},y_j^{-1}]^{\mathbb{C}^*\times\mathbb{C}^*}\]

Computing the invariants gives
\[\mathcal{U}_{ij}=\Spec\C[\frac{x_0}{x_i},\dots,\frac{x_n}{x_i},\frac{y_0^{a_j}}{y_j}.x_i^{\omega_0a_j-\omega_j},\dots].\]

Again powers of $x_i$ appear to make each term invariant under the action of the first coordinate of $(\C^*)^2$ and each $y_k$ comes with a power that is the first number which is 0 modulo $a_j$. In other words, these invariants are exactly the same as those of $\frac{1}{a_j}(0,\dots,0,1,a_1,\dots,a_n)$.
\end{proof}

\subsection{Blow ups of weighted projective space}\label{bl-up-wps}
Example~\ref{F2} already explained the blow up of the weighted projective space $\PP(1,1,1,2)$ at its singular point. In general the rank two toric variety $T$ (or stack, if not well-formed) defined by the homogeneous coordinate ring $\C[y,x_0,\dots,x_n]$ and the irrelevant ideal $I=(y,x_0,\dots,x_k)\cap(x_{k+1},\dots,x_n)$ and the weight system (indicating the action of $(\C^*)^2$)
\[\left(\begin{array}{ccccccc}
\alpha&0&\dots&0&-b_{k+1}&\dots&-b_n\\
0&a_0&\dots&a_k&a_{k+1}&\dots&a_n\end{array}\right),\]
for $1\leq k\leq n-2$ is the (weighted) blow up of the centre $X:(x_{k+1}=\dots=x_n=0)\cong\PP(a_0,\dots,a_k)\subset\PP(a_0,\dots,a_n)$. Details of this constructions are left to the reader to check. A more complicated situation is explained in the next part and the idea and techniques of the proofs there can be applied to this case.

Let us illustrate our method by running the 2-ray game (using VGIT) for the following example. This shows how one can use these explicit methods to study birational geometry of, for example,  hypersurfaces of weighted projective spaces. In fact this problem was considered by Ryder and was partially solved \cite{Ryder}. Applying our techniques, one can classify all elliptic and $K3$ fibrations on Fano threefold hypersurfaces of index\,$1$, together with explicit projective maps that realise such fibration. This problem, however, was later solved with local methods in\,\cite{vanya-jihun}. 

\begin{ex}\label{elliptic}
Consider a Fano threefold $X$ that is a quasi-smooth and general hypersurface of degree $24$ in $\PP=\PP(1,1,6,8,9)$. In a suitable coordinate system, $X$ is defined by $\{f=x_5^2x_3+x_4^3+x_3^4+\cdots+x_1^{24}\}$ in $\PP$. Suppose $p_5=(0:0:0:0:1)$. The germ $p_5\in\PP$ is of type $1/9(1,1,6,8)$ and the germ $p_5\in X$ is a terminal quotient singularity of type $1/9(1,1,8)$. Consider $T$ the blow up of this point in the above description. $T$ is given by a rank two variety with weight system 
\[\left(\begin{array}{cccccc}
3&0&-2&-6&-1&-1\\
0&9&8&6&1&1
\end{array}\right)\]
with the coordinate system $\C[u,x_5,x_4,x_3,x_2,x_1]$ and the irrelevant ideal $I=(u,x_5)\cap(x_4,x_3,x_2,x_1)$. In other words the blow up is given by $(x_5,x_4u^{\frac{2}{3}},x_3u^2,x_2u^{\frac{1}{3}},x_1u^{\frac{1}{3}})$, in coordinates. The restriction of this toric construction to $X$ is the blow up $\hat{X}\rightarrow X$ with weights $1/3(1,1,2)$ blow up of $p_5\in X$. One can check that $T$ is not well-formed and its well-formed model has the weight system
\[\left(\begin{array}{cccccc}
1&3&2&0&0&0\\
0&9&8&6&1&1
\end{array}\right)\]
with respect to which, the equation of $\hat{X}$ has bi-degree $(6,24)$, with equation $x_3x_5^2+x_4^3+x_3^4u^6+\cdots+u^6x_1^3$. Following the 2-ray game of the ambient space (using techniques of~\cite{BZ}), it is easy to verify that $\hat{X}$ forms a fibration over $\mathbb{P}(1,1,6)$ with elliptic fibres $E_{6}\subset\mathbb{P}(1,2,3)$. 
\end{ex}

\subsection{Blow ups of rank 2 toric varieties}
Now we construct Cox rings of rank 3, obtained by blowing up some (torus invariant) centres in a rank 2 toric variety. Then we try to understand the nature of the maps from these varieties to the rank 2 ones. We do this on weighted bundles over $\PP^1$, i.e., when the coordinate ring is $\C[x_0,x_1,y_0,\dots,y_m]$ with irrelevant ideal $I=(x_0,x_1)\cap(y_0,\dots,y_m)$ and weight system
\[\left(\begin{array}{cccccc}
1&1&-\omega_0&-\omega_1&\dots&-\omega_m\\
0&0&1&a_1&\dots&a_m\end{array}\right),\]
for positive integers $a_1,\dots,a_m$ and non-negative integers $\omega_0,\dots,\omega_m$.

It was shown in Proposition~\ref{singularity-F} that each germ \[p_{rs}=\{x_i=y_j=0; \text{ for }i\neq r , j\neq s\}\in\mathcal{U}_{rs}=\{x_rx_s\neq 0\}\]
 has a cyclic quotient singularity of type $\frac{1}{a_j}(0,1,\dots,a_m)$.
Of course this singularity is not isolated. However, instead of blowing up the orbifold locus, we blow up a closed point. The reason for doing this is that we often want to consider the blow up of some subvarieties of $\F$ only at a particular point, see next section for an illustration. We do this by considering the blow up of the ambient space at this point and restrict our attention to the subvariety under this blow up. 

Fix $k\in\{0,\dots,m\}$ and let $T$ be a rank 3 toric variety defined by 
\begin{enumerate}[(i)]
\item $\Cox(T)=\C[X_0,X_1,Y_0,\dots,Y_m,\xi]$,
\item the irrelevant ideal 
\[J=(X_0,X_1)\cap(Y_0,\dots,Y_m)\cap(\xi,X_1)\cap(\xi,Y_k)\cap(X_0,Y_0,\dots,\hat{Y_k},\dots,Y_m)\quad\text{and}\]
\item the action of $(\C^*)^3$ given by the matrix 
\[\left(\begin{array}{ccccccccccc}
1&1&-\omega_0&-\omega_1&\dots&-\omega_{k-1}&-\omega_k&-\omega_{k+1}&\dots&-\omega_m&0\\
0&0&1&a_1&\dots&a_{k-1}&a_k&a_{k+1}&\dots&a_m&0\\
b_k&0&b_0&b_1&\dots&b_{k-1}&0&b_{k+1}&\dots&b_m&-a_k\end{array}\right),\]\end{enumerate}
where $b_0,\dots,b_m$ are strictly positive integers such that
\[b_i\equiv a_i \mod a_k\, \text{ for $i\neq k$ }\qquad \text{ and }\qquad b_k=ra_k\text{ for some positive integer }r.\]

\begin{prop}\label{irrelevant-ideal} The rank 3 toric stack $T$ constructed above is the blow up of the weighted bundle $\F$ over $\PP^1$ in Definition~\ref{weighted bundle-chap3} at the point $(0:1;0:\cdots :0:1:0\cdots :0)$.\end{prop}

\begin{proof} By Proposition~\ref{fan-rank2}, the fan associated to $\F$ consists of $1$-dimensional cones $\alpha_0,\alpha_1$ and $\beta_0\dots,\beta_m$ in $N=\Z^{m+1}$ with $2m+2$ maximal cones \[\sigma_{0i}=\left<\alpha_1,\beta_0\dots,\hat{\beta_i},\dots,\beta_m\right>\text{ and }\sigma_{1j}=\left<\alpha_0,\beta_0\dots,\hat{\beta_j},\dots,\beta_m\right>\text{ for }0\leq i,j\leq m.\]
The last row of the defining matrix of $T$ is clearly adding a new ray in the cone $\sigma_{0k}$. In order to recover the Cox ring of the blow up, we choose the basis of the class group to be that of the original weighted bundle together with the new divisor class associated to this new ray. The fact that the generator of this ray is an integral vector in $N$ is guaranteed by the conditions imposed on $b_i$. This implies that $T$ is the blow up of $\mathcal{F}$ at a point if it has the correct irrelevant ideal. We complete the proof by showing the irrelevant ideal of the Cox ring of this toric blow up is precisely the ideal of $T$. This is done by taking the subdivision of $\sigma_{0k}$ and computing the irrelevant ideal of the new fan using the method of~\cite{cox1}, as in the proof of Proposition~\ref{fan-rank2}. The fan of this blow up of $\Sigma$ consists of rays $\alpha_0,\alpha_1,\beta_0,\dots,\beta_m,\gamma$ and maximal cones
\[\sigma^\prime_{ki}=\left<\alpha_0,\beta_0\dots,\hat{\beta_i},\dots,\hat{\beta_k},\dots,\beta_m,\gamma\right>\text{ for }i\neq k\text{ and }\sigma_k^\prime=\left<\beta_0\dots,\hat{\beta_k},\dots,\beta_m,\gamma\right> \]
coming from the subdivision of $\sigma_{0k}$  together with the remaining cones $\sigma_{ij}$. If we associate the new variable $\xi$ to the ray $\gamma$ and $X_0,X_1$ to $\alpha_0,\alpha_1$ and $Y_i$ to $\beta_i$, then the irrelevant ideal of this toric variety is the ideal generated by
\[X_1\cdot Y_i\cdot Y_k\text{ for all }i\neq k,\quad X_0\cdot X_1\cdot Y_k,\quad X_1\cdot Y_i\cdot\xi\text{ for all }i\neq k,\quad X_0\cdot Y_i\cdot\xi\text{ for all }i.\]
The primary decomposition of this ideal is the irrelevant ideal of $T$.
\end{proof}

\section{Birational models of the del Pezzo fibration}\label{proofs}

Our initial model $X$ is defined as a complete intersection of two hypersurfaces in $\F$, where $\F$ is a $\PP^4$-bundle over $\PP^1$. If we denote by $y_0,y_1,x_0,\dots,x_4$ the global coordinates on $\F$ then the Cox ring of $\F$ is $\Cox(\F)=\C[y_0,y_1,x_0,\dots,x_4]$ with weights
\[\left(\begin{array}{ccccccc}1&1&0&-1&-2&-3&-3\\0&0&1&1&1&1&1\end{array}\right).\]
and irrelevant ideal $I=(y_0,y_1)\cap(x_0,\cdots,x_4)$.

\subsection{Cox ring and description of the initial model}
Let $Q_1\!=\!\{f=0\}$ and $Q_2\!=\!\{g=0\}$, where $f\in|\mathcal{O}_\F(-3,2)|$, i.e. $f$ has bi-degree $(-3,2)$, and $g\in|\mathcal{O}_\F(-2,2)|$. Without loss of generality, after some changes of coordinates, we can assume that $f$ has no monomial term $x_0x_4$, using the symmetry between $x_3$ and $x_4$, and similarly $g$ has no $y_0x_0x_3$ or $y_1x_0x_3$, using the term $x_0x_2$ that appears with non-zero coefficient (see Remark\,\ref{singularity}). In other words $Q_1$, as a divisor, is a general member of the linear system $\mathcal{L}$, where $\mathcal{L}$ is generated by monomials according to the following table
\[\begin{array}{c||c|c|c|c}
\deg \text{ of }y_0, y_1\text{ coefficient}&0&1&2&3\\
\hline
&&&&\\
\text{fibre monomials in }f&x_0x_3&x_2^2&x_2x_3&x_3^2\\ 
&x_1x_2&x_1x_3&x_2x_4&x_3x_4\\
&&x_1x_4&&x_4^2
\end{array}\]     

and similarly $Q_2$ is a general member in the linear system $\mathcal{L}'$ generated by
\[\begin{array}{c||c|c|c|c|c}
\deg \text{ of }y_0, y_1\text{ coefficient}&0&1&2&3&4\\
\hline
&&&&&\\
\text{fibre monomials in }g&x_0x_2&x_0x_4&x_2^2&x_2x_3&x_3^2\\
&x_1^2&x_1x_2&x_1x_3&x_2x_4&x_3x_4\\
&&&x_1x_4&&x_4^2
\end{array}\]     

\begin{rmk}\label{singularity} It is easy to see that $\mathcal{F}$ is smooth. An easy computation (using the table above) shows that the base locus of $\mathcal{L}'$ is the line $\{x_4=x_3=x_2=x_1=0\}$, which is isomorphic to $\PP^1_{y_0:y_1}$ (as $x_0\neq 0$ by the irrelevant ideal of $\mathcal{F}$). Hence by Bertini theorem the singular locus of $Q_2$ lies in this line. But every point in this line in the hypersurface $Q_2:\{g=0\}\subset\mathcal{F}$ is smooth if the monomial $x_0x_2$ appear with nonzero coefficient in $g$. Similarly, $f$ defines a hypersurrface in $\{g=0\}\subset\mathcal{F}$, that is $Q_1\cap Q_2$, and its base locus is $\{g=0\}\cap\{x_4=x_3=x_2=0\}\subset\mathcal{F}$, which corresponds to the same line. Appearance of the monomial $x_0x_3$ in $f$ with nonzero coefficient implies the smoothness of $Q_1\cap Q_2$.\end{rmk}
 Hence with the generality assumption on $f$ and $g$ and the guarantee that $x_0x_2$ in $g$ and $x_0x_3$ in $f$ appear with non-zero coefficients we have:

\begin{lm}\label{smooth} A general threefold $X=Q_1\cap Q_2\subset\F$ as above is smooth.\end{lm}

\subsection{Various 2-ray games} Let us begin by explaining the 2-ray game on $\mathcal{F}$. The GIT chambers of $\mathcal{F}$ can be described as

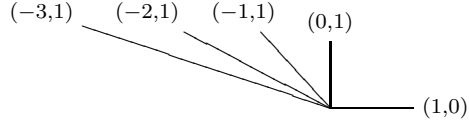
\[\xygraph{
!{(0,0) }="a"
!{(1.2,0) }*+{\scriptstyle{(1,0)}}="b"
!{(0,0.9) }*+{\scriptstyle{(0,1)}}="c"
!{(-1.9,1) }*+{\scriptstyle{(-2,1)}}="d"
!{(-3,1) }*+{\scriptstyle{(-3,1)}}="e"
!{(-0.9,1) }*+{\scriptstyle{(-1,1)}}="f"
"a"-"b"  "a"-"c" "a"-"d" "a"-"e" "a"-"f"}\]
And the 2-ray game, following the VGIT according to this diagram, can be described as
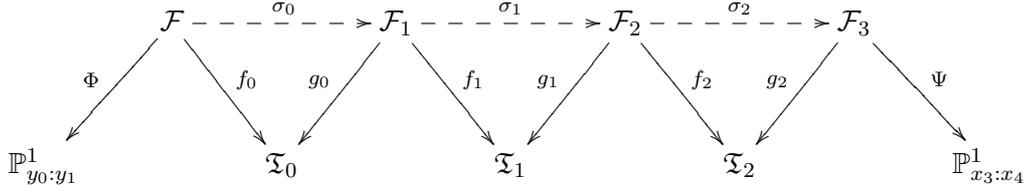
\[\xymatrixcolsep{2pc}\xymatrixrowsep{3pc}
\xymatrix{
&\F\ar_\Phi[ld]\ar^{f_0}[rd]\ar^{\sigma_0}@{-->}[rr]&&\F_1\ar_{g_0}[ld]\ar^{f_1}[rd]\ar^{\sigma_1}@{-->}[rr]&&\F_2\ar_{g_1}[ld]\ar^{f_2}[rd]\ar^{\sigma_2}@{-->}[rr]&&\F_3\ar_{g_2}[ld]\ar^\Psi[rd]&\\
\PP^1_{y_0:y_1}&&\mathfrak{T}_0&&\mathfrak{T}_1&&\mathfrak{T}_2&&\PP^1_{x_3:x_4}
}\]
where the models and the maps are obtained as follows. The varieties in the top row of the diagram (from left to right) correspond to the models constructed as quotient given by a character from the interior of the four chambers in the GIT chambers above (anticlockwise).

In other words, $\F$, $\F_1$, $\F_2$ and $\F_3$ have the same Cox ring and their irrelevant ideals are, in that order,
\[I=(y_0,y_1)\cap(x_0,x_1,x_2,x_3,x_4)\qquad\, I_1=(y_0,y_1,x_0)\cap(x_1,x_2,x_3,x_4)\]
\[I_2=(y_0,y_1,x_0,x_1)\cap(x_2,x_3,x_4)\qquad I_3=(y_0,y_1,x_0,x_1,x_2)\cap(x_3,x_4)\]

Let us demonstrate this for example for $\F_1$. Suppose the chosen character in the interior of the second chamber is $(-1,2)$, then we have that
\[\F_1=\Proj\bigoplus_{n\geq 1}\H^0(\F,n\mathcal{O}_\F(-1,2))=\Proj R_{(-1,2)}\]
where $R_{(-1,2)}$ is generated by monomials of the form
\[\text{(monomial in }y_0,y_1,x_0)\times(\text{monomial in }x_1,x_2,x_3,x_4)\]
so that the total weight is a multiple of $(-1,2)$. Hence the zero graded part, to be removed, corresponds to the ideal $I_1$. 

Similarly, the bottom row of the 2-ray game diagram corresponds to varieties obtained by the one-dimensional rays that separate the GIT chambers (the walls). The two ends of the game are the easiest to calculate. For the other three models we have, for example, that

\[\mathfrak{T}_0=\Proj\bigoplus_{n\geq 1}\H^0(\F,n\mathcal{O}_\F(0,1))=\Proj\C[x_0,y_0x_1,y_1x_1,y_0^2x_2,\dots,y_1^3x_4]\]
and hence the map $f_0$ is given in coordinates by
\[(y_0,y_1,x_0,\dots,x_4)\mapsto(x_0,y_0x_1,y_1x_1,y_0^2x_2,\dots,y_1^3x_4)\]
It is not hard to see that $f_0$ is one-to-one away from the point $p_0=(1:0:\dots:0)\in\mathfrak{T}_0$. On the other hand $f_0^{-1}(p_0)$ is given by the solutions of 
\[y_0x_1=y_1x_1=y_0^2x_2=\dots=y_1^3x_4=0\]
which is the set
\[\{y_0=y_1=0\}\cup\{x_1=\dots=x_4=0\}\]

The set $\{y_0=y_1=0\}$ is a component of $I$, and hence $f_0^{-1}(p_0)=\{x_1=\dots=x_4=0\}\subset\F$, which is isomorphic to $\PP^1$. We say that $f_0$ contracts this line to $p_0$. Similarly $g_0$ contracts the locus $\{y_0=y_1=0\}$, which is isomorphic to $\PP(1,2,3,3)$, to $p_0$. This could also be viewed by restricting our observation to the open set $U:(x_0\neq 0)\subset\mathfrak{T}_0$, which is
\[U=\Spec\C[y_0x_1,y_1x_1,y_0^2x_2,\dots,y_1^3x_4]\]
Note that $p_0$ corresponds to $O\!\in\!U$. Setting $x_0=1$ we can remove the action of the second component of $(\C^*)^2$, i.e., the action is $(1,1,-1,-2,-3,-3)$. Following the notation of \cite{reid} $\sigma_0\colon\F\dashrightarrow\F_1$ is an anti-flip of type
$(1,1,-1,-2,-3,-3)$, that is, a $\PP^1\subset\F$ is replaces by a $\PP(1,2,3,3)\subset\F_1$.

Similarly $\sigma_1\colon\F_1\dashrightarrow\F_2$ is an anti-flip of type $(1,1,1,-1,-2,-2)$, and it replaces $\PP^2\subset\F_1$ by $\PP(1,2,2)\subset\F_2$. Let us explain where these numbers came from. Ideally, we would like to set $x_1$ (the variable at the base of the flip or anti-flip) to be $1$ and remove the second component of the action of $(\C^*)^2$. This can be done if this variable has no effect on the action of the first component (exactly like the previous case for $x_0$). But this can be done by multiplying the matrix of weights of the Cox ring of $\F$ by
\[\left(\begin{array}{cc}1&1\\0&1\end{array}\right)\in\Gl(2,\Z)\]
and recovering the anti-flip type $(1,1,1,-1,-2,-2)$. By anti-flip we mean the inverse of a flip. Whether such map is a flip or anti-flip, or a flop, depends on the sign of the sum of the numbers in the description. For example, it can be seen similarly that $\sigma_2$ is a flip of type $(1,1,2,1,-1,-1)$ as the sum of these numbers is positive. The structure of the other maps in the 2-ray game can be verified similarly, and we leave the details to the reader.

\subsubsection{The restriction of the 2-ray game of $\F$ to $Q_1$}
Let us recall that the fourfold $Q_1$ is defined as the zero locus of the polynomial $f=x_0x_3+x_1x_2+\dots$ in $\F$. The restriction of $\sigma_0$ to $Q_1$ is an anti-flip of type $(1,1,-1,-2,-3)$. This can be seen by observing that in the loci $x_0\neq 0$ (in either side of the anti-flip of $\sigma_0$) $f$ has a linear term $x_3$, and therefore this variable can be eliminated, hence the anti-flip above, which we denote by $\sigma_0^1\colon Q_1\dashrightarrow Q_1^1$, where $Q_1^1$ is the image of $Q_1$ under $\sigma_0$. Similarly $\sigma_1$ restricted to $Q_1^1$ is an anti-flip $\sigma_1^1\colon Q_1^1\dashrightarrow Q_1^2$ of type $(1,1,1,-2,-2)$ as the variable $x_2$ can be eliminated by setting $x_1\neq 0$. Similarly we have a flip $\sigma_2^1\colon Q_1^2\dashrightarrow Q_1^3$ of type $(1,1,2,-1,-1)$. The restriction of $\Psi$ to $Q_1^3$ is a fibration over $\PP^1$ with fibres given by a degree $3$ polynomial in $\PP(1,1,1,2,3)$. A general fibre (as long as $x_3\neq 0$) is isomorphic to $\PP(1,1,1,2)$.

\subsubsection{The restriction of the 2-ray game of $Q_1$ to $X$}\label{X-to-Y}
Similar to the previous case, we restrict the 2-ray game of $Q_1$ to $X$. We have the following diagram.

\[\xymatrixcolsep{2pc}\xymatrixrowsep{3pc}
\xymatrix{
&X\ar_\varphi[ld]\ar[rd]\ar^{\sigma^2_0:(1,1,-1,-3)}@{-->}[rr]&&X_1\ar[ld]\ar[rd]\ar^{\sigma^2_1:\text{ isomorphism}}@{->}[rr]&&X_2\ar[ld]\ar[rd]\ar^{\sigma^2_2:(1,1,-1,-1)}@{-->}[rr]&&Y\ar[ld]\ar^\psi[rd]&\\
\PP^1_{y_0:y_1}&&Z_0&&Z_1&&Z_2&&\PP^1_{x_3:x_4}
}\]
The map $\sigma_0^2$ is the restriction of $\sigma_0^1$ to $X$. Because of the appearance of the monomial $x_0x_2$ with nonzero coefficient in $g$, the variable $x_2$ can be eliminated near the flipped locus. Similarly, for $\sigma_2^2$ the variable $x_0$ is eliminated to get $(1,1,-1,-1)$. The flip $\sigma_1^1$ happens away from $X_1$, as the monomial $x_1^2$ in $g$. Note that $(3,1,-1,-1)$ is a terminal flip (see~\cite{danilov, gavin}) and $\sigma_2^2$ is the Atiyah flop, so all models $X$, $X_1$, $X_2$ and $X_3$ have terminal singularities. The fibres of $\varphi$ are intersections of two quadrics in $\PP^4$, hence $X$ is a del Pezzo surface of degree $4$, and fibres of $\psi$ are quartic surfaces in $\PP(1,1,1,2)$, that is, del Pezzo surfaces of degree $2$.

The threefold $Y$ is defined by $\{f=g=0\}$ in the toric variety $\F_3$ with $\Cox(\F_3)=\C[u,v,x,y,z,t,s]$, with weights
\[\left(\begin{array}{ccccccc}1&1&0&-1&-2&-1&-1\\0&0&1&2&3&1&1\end{array}\right),\]
and irrelevant ideal $I_Y=(u,v)\cap(x,y,z,t,s)$. 

Note that, for simplicity, we have renamed the variables, i.e., the variables $u,v,x,y,z,t,s$ are exactly $x_4,x_3,x_2,x_1,x_0,y_0,y_1$ (in that order). The weight matrix is that of $\F$ in opposite order, in rays and columns, multiplied by a matrix 
\[\left(\begin{array}{cc}-2&-1\\3&1\end{array}\right)\in\Sl(2,\Z).\]
The equations of $f$ and $g$ must be easily read after the substitution, in particular, 
\[f=0.vz+uz+x^2t+\cdots\quad\text{ and }\quad g=0.vzt+0.vzs+y^2+xz+\cdots\]
with bi-degrees $(-1,3)$ and $(-2,4)$, with respect to the new weights. By Proposition~\ref{singularity-F}, $\F_1$, the ambient toric variety in which $Y$ is embedded, has two lines of singularities of types $\A^1\times 1/2(1,1,1,1)$ and $\A^1\times 1/3(1,1,1,2)$. It is easy to check that $Y$ does not meet the first line and it intersects the second line at the point $p_{vz}=(u=x=y=t=s=0)$. In particular, $p_{vz}$ is a terminal singularity of type $1/3(1,1,2)$, as one can eliminate the variables $u$ and $x$ in an analytical neighbourhood of this point using the local descriptions of $f$ and $g$.

\begin{theorem} The threefold $X$ is a del Pezzo fibration of degree $4$ and it is birational to a del Pezzo fibration of degree $2$ (over $\PP^1$).
\end{theorem}
\begin{proof} How the birational map is obtained was already explained above in\,\ref{X-to-Y}. The only part left to prove is that $\Pic(X)\cong\Z^2$. Note that $Q_1$ is an ample divisor on $\F_2$. We know that $\Pic(\F)\cong\Z^2$. It follows from Lefschetz hyperplane theorem and the fact that $\F$ and $\F_2$, and similarly $Q_1$ and $Q_1^2$, are isomorphic in codimension 1 that $\Pic(Q_1)\cong\Z^2$.  The same argument will not work instantly for $X$ as it is not ample on $Q_1$, and neither is any of the other models isomorphic to $X$ in codimension 1 ample on any of the $Q_1^i$s. However, $X$ is nef and big for two models $Q_1^1$ and $Q_1^2$, and this is enough to use Lefschetz type argument as explained in similar cases in \cite[\S4.3]{Ahm} in order to obtain $\Pic(X)\cong\Z^2$.
\end{proof}

From the construction of the 2-ray games above it is clear that $X$ is a Mori dream space. In fact, varieties $X$, $X_1$, $X_2$ and $Y$ are all birational models isomorphic in codimension $1$ and as the Picard number is $2$ it follows that the boundaries of the mobile cone, of $X$ or the other models, correspond to the divisors that give the two fibrations over the projective line (from $X$ and $Y$). The following conjecture is a very natural consequence of this phenomenon.

\begin{conj} The Cox ring of the threefold  Mori dream space $X$ is $\Cox(X)=\dfrac{\Cox(\F)}{(f,g)}$.\end{conj}

By rules of Sarkisov program, see \cite{Corti1} \S2.2, there are two possibilities for the start of the program: having a Mori fibre space $X\rightarrow S$, either run one step of the MMP on $S$, obtain $S\dashrightarrow T$, and consider $\Pic(X/T)$ or do a blow up on $X$, obtain $Z\rightarrow X$, and consider $\Pic(Z/S)$. In both cases the relative Picard group has rank 2. One generator of this group corresponds to $Z\rightarrow X$ or $S\dashrightarrow T$ and the other generator indicates the beginning of the so-called 2-ray game, and in correct setting the Sarkisov link. See \cite{Corti1} for details. What we did on $X$ (similarly on $Q$ and $\F$) is of the first type. Other examples can be found in~\cite{BCZ, Ahm}. In Subsection~\ref{rank3} we perform Sarkisov links of the other type, i.e., starting with a blow up.

\subsection{2-ray game starting by a blow up}\label{rank3}

We want to show that a blow up of the point $p_{vz}$ is the start of a Sarkisov link on $Y$. By  \cite{kawamata}, in order to remain in the Mori category, this blow up is the unique weighted blow up of type $(1,1,2)$ with discrepancy $\frac{1}{3}$; this is called the {\it Kawamata blow up} of the point $p_{vz}$~\cite[Definition~3.4.2]{CPR}. In other words, if we denote the blow up of $Y$ by $\mathcal{Y}$, and the exceptional divisor of the blow up by $E$, then
\[K_\mathcal{Y}=K_Y+\frac{1}{3}E\]

Define a rank 3 toric Cox ring by $\mathcal{R}=\C[u,v,x,y,z,t,s,w]$, the irrelevant ideal $\mathcal{I}=(u,v)\cap(x,y,z,t,s)\cap(u,x,y,t,s)\cap(w,v)\cap(w,z)$ and weights given by the matrix
\[\mathcal{A}=\left(\begin{array}{cccccccc}
1&1&0&-1&-2&-1&-1&0\\0&0&1&2&3&1&1&0\\3&0&a&2&0&1&1&-3
\end{array}\right).\]
By Proposition~\ref{irrelevant-ideal}, $\mathcal{T}=(\mathcal{I},\mathcal{A})$ is a blow up of $\mathcal{F}_3$, for a positive integer $a=3k+1$. Note that in terms of the fan structure of $\mathcal{T}$ we have added the ray $r$ to the fan of $\F_3$, that is the ray $r$ so that
\[3r=3r_u+ar_x+2r_y+r_t+r_s\]
where every ray in the fan of $\F_3$ is indexed by the corresponding variable from the Cox ring.

The aim is to show that for a particular $k$, the proper transform of $Y$ under this blow up ($\mathcal{Y}$) is the Kawamata blow up of $Y$ at $p_{vz}$ and the relative 2-ray game of $\mathcal{T}$ restricts to a relative 2-ray game on $\mathcal{Y}$.

Note that the blow up map is given by
\[\pi((u,v,x,y,z,t,s,w))\mapsto(uw,v,w^\frac{a}{3}x,w^\frac{2}{3}y,z,w^\frac{1}{3}t,w^\frac{1}{3}s)\]

The well-formed model of $\mathcal{A}$ is (denoted by $\mathcal{A}$ again)
\[\mathcal{A}=\left(\begin{array}{cccccccc}
1&1&0&-1&-2&-1&-1&0\\
0&0&1&2&3&1&1&0\\
1&0&k&0&-1&0&0&-1
\end{array}\right).\]

\begin{lm} For $a=4$, $\mathcal{Y}\longrightarrow Y$ is the unique Kawamata blow up of the point $p_{vz}\in Y$, where $\mathcal{Y}\subset\mathcal{T}$ is the birational transform of $Y$ by $\varphi$.\end{lm}
\begin{proof} Since $Y$ is a complete intersection in $T$, we have that $-K_Y=(-K_T-Y)|_Y$, and similar formula for $\mathcal{Y}\subset\mathcal{T}$. If we denote by $\tilde{f}$ and $\tilde{g}$ the defining equations of $\mathcal{Y}\subset\mathcal{T}$, then $\deg\tilde{f}=(-1,3,0)$ and $\deg\tilde{g}=(-2,4,0)$. In particular, 
\[-K_Y\sim\mathcal{O}_Y((0,1))\quad\text{and}\quad
-K_\mathcal{Y}\sim\mathcal{O}_\mathcal{Y}((0,1,k-1))\]
In other words, $K_\mathcal{Y}\sim -D_x-E$, where $D_x=(x=0)$ is a principal divisor on $\mathcal{Y}$ and $E=(w=0)$ is the exceptional divisor. On the other hand, $\pi^*(K_Y)\sim -D_x-\frac{a}{3}E$. Using Kawamata's condition that the discrepancy is equal to $\frac{1}{3}$, we conclude that $a=4$ (and $k=1$).
\end{proof}

\subsubsection*{\bf The fan of $\mathcal{T}$} The description of the fan of $\F_3$ is as follows. It has $7$ 1-dimensional cones (rays) generated by 
\[r_u=( 1,  0,  0,  0,  0),  r_v=( 0,  1,  0,  0,  0), r_x=( 0,  0,  1,  0,  0)\]
\[r_y=( 0,  0,  0,  1,  0), r_z=(-1, -1, -1, -1,  0), r_t=( 0,  0,  0,  0,  1), r_s=( 3,  3,  2,  1, -1)\]

in $\Z^5$ and ten maximal cones
 \[ C_{vs}= \left<r_u, r_x, r_y,r_z,r_t\right>,\quad \dots \quad ,C_{ux}=\left<r_v,r_y,r_z,r_t,r_s\right>\]
 
 The fan of $\mathcal{T}$ then is obtained by adding the ray $\left<r_w\right>$ such that 
 \[r_u+r_x=r_z+r_w\]
 that is the point $(2,1,2,1,0)$, which is indeed 
 \[ \frac{1}{3}(6,3,6,3,0)=\frac{1}{3}(3r_u+4r_x+2r_y+r_s+r_t).\]
Hence it is in the interior of the cone $C_{vz}$. The maximal cones of the fan of $\mathcal{T}$ are obtained from those of $\F_3$ with $C_{vz}$ replaced by its subdivision through $r_w$.

\begin{lm} $Y$ is square birational to a degree 2 del Pezzo model embedded in a toric variety as a hypersurface.\end{lm}
\begin{proof} Let us first consider the toric 2-ray game corresponding to the relative Picard number $\rank\Pic(\mathcal{T})-\rank\Pic(\PP^1_{u:v})=2$. Using toric MMP we can see that this 2-ray game on $\mathcal{T}$ consists of a flop to $\mathcal{T}^\prime$ followed by a divisorial contraction to a rank 2 toric variety $\F'$. Let us explain this in detail.

As we blew up a point in a fibre of the fibration (over $\PP^1$) in $\F_3$, we can turn our attention to the open set $U_v:(u\neq 0)\subset\PP^1_{u:v}$. By doing so, and removing the action of the first component of $(\C^*)^3$, we have a rank two toric variety $\mathcal{T}_v$ with irrelevant ideal $\mathcal{I}_v=(u,x,t,s,t)\cap(z,w)$ and grading
\[\left(\begin{array}{ccccccc}
0&1&1&1&2&3&0\\
1&1&0&0&0&-1&-1
\end{array}\right)\]
where the columns (from left to right) indicate the weight of the variables $u,x,t,s,y,z,w$. The 2-ray game now follows by a flop of type $(1,1,-1,-1)$ over $\PP(1,1,2)$, that is, a $\PP^1$-bundle over $\PP(1,1,2)$ (the locus $\{u=x=0\}$) contracts to the base $\PP(1,1,2)$ on one side of the flop and on the other side there is another $\PP^1$-bundle (the locus $\{z=w=0\}$), over the same base, is extracted. Note that every time we ran the 2-ray game before it was ``from left to right'' with respect to our matrix of weights but this time it was from ``right to left''. 

In terms of Cox ring of $\mathcal{T}'$, this flop is the replacement of $\mathcal{I}$ by another ideal $\mathcal{I^\prime}=(u,v)\cap(w,v)\cap(u,x)\cap(w,y,z,t,s)\cap(x,y,z,t,s)$. 

 We can rewrite $\mathcal{A}$, after a $\Sl(3,\Z)$ modification, as
\[\mathcal{A}=\left(\begin{array}{cccccccc}
1&1&0&-1&-2&-1&-1&0\\0&0&1&2&3&1&1&0\\1&0&0&-2&-4&-1&-1&-1
\end{array}\right).\]
This is done to see (the weights for) the divisorial contraction on $\mathcal{T}'$; indicated by the last row.
By Proposition~\ref{irrelevant-ideal}, $\pi'\!\colon\!\mathcal{T'}\!\rightarrow\!\F'$ is a divisorial contraction with exceptional divisor $E'\!=\!(u=0)$ to a point, where $\Cox(\F^\prime)=\C[w,v,z,y,x,t,s]$, with irrelevant ideal $I'=(w,v)\cap(z,y,x,t,s)$ and the weights
\[\mathcal{A'}=\left(\begin{array}{ccccccc}
1&1&-1&-1&-1&-1&-1\\0&0&3&2&1&1&1
\end{array}\right)\]
for the variables $w,v,z,y,x,t,s$, in that order from left to right on the columns of $\mathcal{A'}$.

Note that in terms of the fan structure, we have just removed the ray $r_u$ that is interior to the cone $C_{vx}$; this can be seen from the relation in the last row of $\mathcal{A}$.

The equation of this contraction can be written as
\[\pi'((u,v,x,y,z,t,s,w))\mapsto(uw,v,u^4z,u^2y,x,ut,us),\]
in particular, the image of the contraction is the point $p_{vx}\in\F'$.

Now let us restrict this game to $\mathcal{Y}$. We see that the toric flop of the ambient space restricts to a flop of 6 lines. This is because the two equations $\tilde{f}$ and $\tilde{g}$ restricted to the base of the flop, the $\PP(1,1,2)$, define a cubic and a quartic equation, respectively, which is $6$ points \cite[Lemma~9.5]{fle} (the base of the three dimensional flop).

Setting $(u=0)$, to recover the three dimensional exceptional divisor of the contraction, we can set $u\neq 0$ and $x\neq 0$ (using $\mathcal{I}'$). After appropriate elimination using $xz$ in $\tilde{g}$ we observe that the exceptional divisor is a cubic (what remain from $\tilde{f}$ after setting $u=0$ and $v=x=1$) in $\PP(1,1,2)$. The birational transforms of $f$ and $g$ in $\F'$ are, respectively 
\[\hat{f}=z+wxy+vyt+\cdots\quad\text{ and }\quad \hat{g}=y^2+xz+\cdots\]
In particular, we can eliminate $z$ globally, and consider $Y'$, the birational transform of $Y$, embedded as a hypersurface in a rank 2 toric variety with Cox ring equal to that of $\F'$, with $z$ removed. The threefold $Y'$ is the vanishing of the polynomial $y^2=vyxt+w^2t^4+\cdots$. It is easy to check that $Y'$ is a fibration of degree 2 del Pezzo surfaces over $\PP^1_{v:w}$, and is square birational to $Y$, over $\PP^1$.
\end{proof}

\begin{lm} $Y'$ is birational to a conic bundle.\end{lm}
\begin{proof} This appears as Family~6 in Theorem~3.3 of~\cite{Ahm}.\end{proof}
Diagram\,\ref{main-diag} captures the geometry of $X$ and its birational models constructed in this section.

{\small
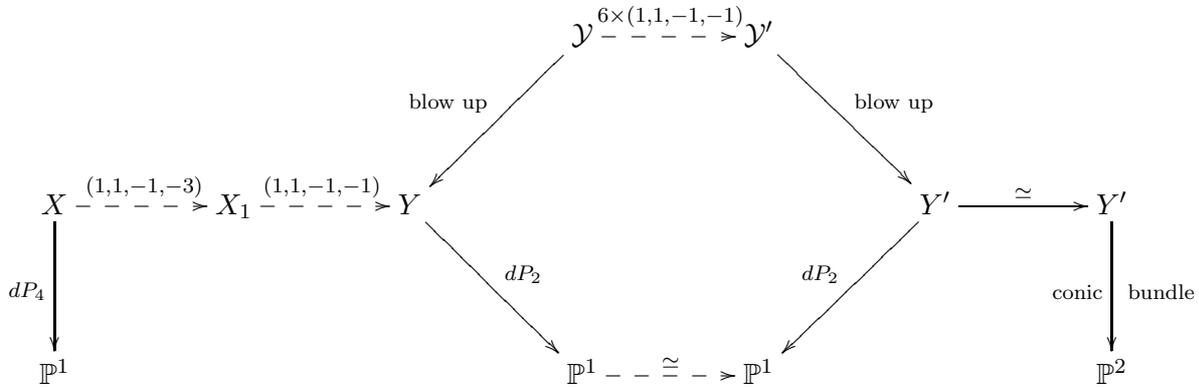
\begin{table}[h]\centering 

\xymatrixcolsep{4pc}\xymatrixrowsep{4pc}
\xymatrix{
&&&\mathcal{Y}\ar_{\text{blow up}}[ld]\ar@{-->}^{6\times(1,1,-1,-1)}[r]&\mathcal{Y}^\prime\ar^{\text{blow up}}[rd]&&\\
X\ar@{-->}^{(1,1,-1,-3)}[r]\ar_{dP_4}[d]&X_1\ar@{-->}^{(1,1,-1,-1)}[r]&Y\ar^{dP_2}[rd]&&&Y'\ar_{dP_2}[ld]\ar^{\simeq}[r]&Y'\ar_{\text{conic}}^{\text{ bundle}}[d]\\
\PP^1&&&\PP^1\ar@{-->}^{\simeq}[r]&\PP^1&&\PP^2
}
\caption{}\label{main-diag}
\end{table} 
}

In order to complete the proof of Theorem~\ref{main-theorem} we need to show that $X$ is not rational. To obtain this, we use the result of Alexeev~\cite{alexeev}, whose proof relies on Shokurov's work \cite{Shokurov} on nonrationality of standard conic bundles over $\mathbb{F}_n$.

\begin{theorem}[\cite{alexeev}, Theorem~2]\label{alexeev} If the Euler characteristic of a (standard) del Pezzo fibration of degree 4 over $\PP^1$ does not belong to $\{0,-4,-8\}$, then it is not rational.
\end{theorem}

Let us recall that a fibration is standard if it has Picard number two and all the fibres are normal. Both these conditions are satisfied with the models considered in this article, as verified above.

\begin{lm} $X$ is not rational.\end{lm}
\begin{proof} By the same calculation as \cite{ful}~Example~3.2.11 we can compute $\chi(X)=-28$, the Euler characteristic of $X$, see\,\cite{shramov} for many similar cases. The result follows from Theorem~\ref{alexeev}.\end{proof}

\bibliographystyle{amsplain}
\bibliography{bib}

\providecommand{\bysame}{\leavevmode\hbox to3em{\hrulefill}\thinspace}
\providecommand{\MR}{\relax\ifhmode\unskip\space\fi MR }
\providecommand{\MRhref}[2]{%
  \href{http://www.ams.org/mathscinet-getitem?mr=#1}{#2}
}
\providecommand{\href}[2]{#2}
\begin{thebibliography}{10}

\bibitem{Ahm}
Hamid Ahmadinezhad, \emph{On del {P}ezzo fibrations that are not birationally
  rigid}, J. Lond. Math. Soc. \textbf{86} (2012), no.~1, 36--62.

\bibitem{alexeev}
V.~A. Alekseev, \emph{On conditions for the rationality of three-folds with a
  pencil of del {P}ezzo surfaces of degree {$4$}}, Mat. Zametki \textbf{41}
  (1987), no.~5, 724--730, 766. \MR{898133 (88j:14048)}

\bibitem{Hausen}
Florian Berchtold and J{\"u}rgen Hausen, \emph{Cox rings and combinatorics},
  Trans. Amer. Math. Soc. \textbf{359} (2007), no.~3, 1205--1252 (electronic).
  \MR{2262848 (2007h:14007)}

\bibitem{borisov}
Lev~A. Borisov, Linda Chen, and Gregory~G. Smith, \emph{The orbifold {C}how
  ring of toric {D}eligne-{M}umford stacks}, J. Amer. Math. Soc. \textbf{18}
  (2005), no.~1, 193--215 (electronic). \MR{2114820 (2006a:14091)}

\bibitem{gavin}
Gavin Brown, \emph{Flips arising as quotients of hypersurfaces}, Math. Proc.
  Cambridge Philos. Soc. \textbf{127} (1999), no.~1, 13--31. \MR{1692523
  (2000f:14018)}

\bibitem{BCZ}
Gavin Brown, Alessio Corti, and Francesco Zucconi, \emph{Birational geometry of
  3-fold {M}ori fibre spaces}, The {F}ano {C}onference, Univ. Torino, Turin,
  2004, pp.~235--275. \MR{2112578 (2005k:14031)}

\bibitem{BZ}
Gavin Brown and Francesco Zucconi, \emph{Graded rings of rank 2 {S}arkisov
  links}, Nagoya Math. J. \textbf{197} (2010), 1--44. \MR{2649280
  (2011f:14020)}

\bibitem{vanya-rigid}
Ivan Cheltsov, \emph{Birationally rigid del {P}ezzo fibrations}, Manuscripta
  Math. \textbf{116} (2005), no.~4, 385--396. \MR{2140210 (2006b:14067)}

\bibitem{Cheltsov}
\bysame, \emph{Birationally rigid del {P}ezzo fibrations}, Manuscripta Math.
  \textbf{116} (2005), no.~4, 385--396. \MR{2140210 (2006b:14067)}

\bibitem{vanya-jihun}
Ivan Cheltsov and Jihun Park, \emph{Weighted {F}ano threefold hypersurfaces},
  J. Reine Angew. Math. \textbf{600} (2006), 81--116. \MR{2283799
  (2007k:14019)}

\bibitem{corti2}
Alessio Corti, \emph{Factoring birational maps of threefolds after {S}arkisov},
  J. Algebraic Geom. \textbf{4} (1995), no.~2, 223--254. \MR{1311348
  (96c:14013)}

\bibitem{Corti1}
\bysame, \emph{Singularities of linear systems and {$3$}-fold birational
  geometry}, Explicit birational geometry of 3-folds, London Math. Soc. Lecture
  Note Ser., vol. 281, Cambridge Univ. Press, Cambridge, 2000, pp.~259--312.
  \MR{1798984 (2001k:14041)}

\bibitem{corti-mella}
Alessio Corti and Massimiliano Mella, \emph{Birational geometry of terminal
  quartic 3-folds. {I}}, Amer. J. Math. \textbf{126} (2004), no.~4, 739--761.
  \MR{2075480 (2005d:14019)}

\bibitem{CPR}
Alessio Corti, Aleksandr Pukhlikov, and Miles Reid, \emph{Fano {$3$}-fold
  hypersurfaces}, Explicit birational geometry of 3-folds, London Math. Soc.
  Lecture Note Ser., vol. 281, Cambridge Univ. Press, Cambridge, 2000,
  pp.~175--258. \MR{1798983 (2001k:14034)}

\bibitem{Cox}
David~A. Cox, \emph{The homogeneous coordinate ring of a toric variety}, J.
  Algebraic Geom. \textbf{4} (1995), no.~1, 17--50. \MR{1299003 (95i:14046)}

\bibitem{cox1}
\bysame, \emph{What is a toric variety?}, Topics in algebraic geometry and
  geometric modeling, Contemp. Math., vol. 334, Amer. Math. Soc., Providence,
  RI, 2003, pp.~203--223. \MR{2039974 (2005b:14089)}

\bibitem{Cox-book}
David~A. Cox, John~B. Little, and Henry~K. Schenck, \emph{Toric varieties},
  Graduate Studies in Mathematics, vol. 124, American Mathematical Society,
  Providence, RI, 2011. \MR{2810322 (2012g:14094)}

\bibitem{danilov}
V.~I. Danilov, \emph{Birational geometry of three-dimensional toric varieties},
  Izv. Akad. Nauk SSSR Ser. Mat. \textbf{46} (1982), no.~5, 971--982, 1135.
  \MR{675526 (84e:14008)}

\bibitem{defe}
Tommaso de~Fernex, \emph{Birationally rigid hypersurfaces}, Invent. Math.
  \textbf{192} (2013), no.~3, 533--566. \MR{3049929}

\bibitem{dol}
Igor Dolgachev, \emph{Weighted projective varieties}, Group actions and vector
  fields ({V}ancouver, {B}.{C}., 1981), Lecture Notes in Math., vol. 956,
  Springer, Berlin, 1982, pp.~34--71. \MR{704986 (85g:14060)}

\bibitem{fantechi}
Barbara Fantechi, Etienne Mann, and Fabio Nironi, \emph{Smooth toric
  {D}eligne-{M}umford stacks}, J. Reine Angew. Math. \textbf{648} (2010),
  201--244. \MR{2774310 (2012b:14097)}

\bibitem{ful}
William Fulton, \emph{Intersection theory}, second ed., Ergebnisse der
  Mathematik und ihrer Grenzgebiete. 3. Folge. A Series of Modern Surveys in
  Mathematics [Results in Mathematics and Related Areas. 3rd Series. A Series
  of Modern Surveys in Mathematics], vol.~2, Springer-Verlag, Berlin, 1998.
  \MR{1644323 (99d:14003)}

\bibitem{Grinenko1}
M.~M. Grinenko, \emph{Birational properties of pencils of del {P}ezzo surfaces
  of degrees 1 and 2}, Mat. Sb. \textbf{191} (2000), no.~5, 17--38. \MR{1773767
  (2001g:14016)}

\bibitem{Grinenko2}
\bysame, \emph{Birational properties of pencils of del {P}ezzo surfaces of
  degrees 1 and 2. {II}}, Mat. Sb. \textbf{194} (2003), no.~5, 31--60.
  \MR{1992109 (2004i:14012)}

\bibitem{Grinenko}
\bysame, \emph{Fibrations into del {P}ezzo surfaces}, Uspekhi Mat. Nauk
  \textbf{61} (2006), no.~2(368), 67--112. \MR{2261543 (2007g:14009)}

\bibitem{hacon-mcK}
Christopher~D. Hacon and James McKernan, \emph{The {S}arkisov program}, J.
  Algebraic Geom. \textbf{22} (2013), no.~2, 389--405. \MR{3019454}

\bibitem{hu}
Yi~Hu and Sean Keel, \emph{Mori dream spaces and {GIT}}, Michigan Math. J.
  \textbf{48} (2000), 331--348, Dedicated to William Fulton on the occasion of
  his 60th birthday. \MR{1786494 (2001i:14059)}

\bibitem{fle}
A.~R. Iano-Fletcher, \emph{Working with weighted complete intersections},
  Explicit birational geometry of 3-folds, London Math. Soc. Lecture Note Ser.,
  vol. 281, Cambridge Univ. Press, Cambridge, 2000, pp.~101--173. \MR{1798982
  (2001k:14089)}

\bibitem{manin}
V.~A. Iskovskih and Ju.~I. Manin, \emph{Three-dimensional quartics and
  counterexamples to the {L}\"uroth problem}, Mat. Sb. (N.S.) \textbf{86(128)}
  (1971), 140--166. \MR{0291172 (45 \#266)}

\bibitem{isk}
V.~A. Iskovskikh, \emph{Factorization of birational mappings of rational
  surfaces from the point of view of {M}ori theory}, Uspekhi Mat. Nauk
  \textbf{51} (1996), no.~4(310), 3--72. \MR{1422227 (97k:14016)}

\bibitem{kawamata}
Yujiro Kawamata, \emph{Divisorial contractions to {$3$}-dimensional terminal
  quotient singularities}, Higher-dimensional complex varieties ({T}rento,
  1994), de Gruyter, Berlin, 1996, pp.~241--246. \MR{1463182 (98g:14005)}

\bibitem{Pukhlikov}
A.~V. Pukhlikov, \emph{Birational automorphisms of three-dimensional algebraic
  varieties with a pencil of del {P}ezzo surfaces}, Izv. Ross. Akad. Nauk Ser.
  Mat. \textbf{62} (1998), no.~1, 123--164. \MR{1622258 (99f:14016)}

\bibitem{reid}
Miles Reid, \emph{What is a flip?}, unpublished manuscript of {U}tah seminar
  (1992).

\bibitem{M1}
\bysame, \emph{Chapters on algebraic surfaces}, Complex algebraic geometry
  ({P}ark {C}ity, {UT}, 1993), IAS/Park City Math. Ser., vol.~3, Amer. Math.
  Soc., Providence, RI, 1997, pp.~3--159. \MR{1442522 (98d:14049)}

\bibitem{Ryder}
Daniel Ryder, \emph{Classification of elliptic and {$K3$} fibrations birational
  to some {$\Bbb Q$}-{F}ano 3-folds}, J. Math. Sci. Univ. Tokyo \textbf{13}
  (2006), no.~1, 13--42. \MR{2223680 (2007b:14029)}

\bibitem{Shokurov}
V.~V. Shokurov, \emph{Distinguishing {P}rymians from {J}acobians}, Invent.
  Math. \textbf{65} (1981/82), no.~2, 209--219. \MR{641127 (83f:14028)}

\bibitem{shramov}
K.~A. Shramov, \emph{On the rationality of nonsingular threefolds with a pencil
  of del {P}ezzo surfaces of degree 4}, Mat. Sb. \textbf{197} (2006), no.~1,
  133--144. \MR{2230135 (2007h:14015)}

\end{thebibliography}

\vspace{0.01cm}

Radon Institute, Austrian Academy of Sciences

Altenberger Str. 69, A-4040 Linz, Austria

e-mail: \url{hamid.ahmadinezhad@oeaw.ac.at}

\end{document}